\documentclass[13pt,a4paper, oneside]{amsart}       
\usepackage{fullpage}
\usepackage{CJK}
\usepackage{amssymb}
\usepackage{graphicx}
\allowdisplaybreaks[4]
\usepackage{mathrsfs}
\usepackage{amscd}
\usepackage[all]{xy}           
\usepackage{amsfonts,latexsym}
\usepackage[dvipsnames]{xcolor}
\usepackage{amsthm}
\usepackage{ytableau}
\usepackage{tikz}
\usepackage{ifpdf}\ifpdf
\usepackage[colorlinks=true,citecolor=blue,final,backref=page,hyperindex]{hyperref}\else
\usepackage[colorlinks=true,citecolor=blue,final,backref=page,hyperindex,hypertex]{hyperref}\fi
\xyoption{all}
\allowdisplaybreaks

\newenvironment{Young}{\begingroup
       \def\vr{\vrule height0.8\hoogte width\dikte depth 0.2\hoogte}
       \def\fbox##1{\vbox{\offinterlineskip
                    \hrule height\dikte
                    \hbox to \breedte{\vr\hfill$##1$\hfill\vr}
                    \hrule height\dikte}}
       \vtop\bgroup \offinterlineskip \tabskip=-\dikte \lineskip=-\dikte
            \halign\bgroup &\fbox{##\unskip}\unskip  \crcr}
     {\egroup\egroup\endgroup}

\def\diagram(#1){\,\vcenter{\begin{Young}#1\cr\end{Young}}}

\fontsize{9pt}{9pt}\selectfont

\def\bgl{\boldsymbol{\lambda}}

\def\Number#1{\refstepcounter{equation}
              \leqno(\theequation)\if*#1%
              \else\def\@currentlabel{{\rm\theequation}}\label{#1}%
              \fi}

\def\lexp#1#2{\kern\scriptspace\vphantom{#2}^{#1}\kern-\scriptspace#2}

\def\@underbar#1#2{\settowidth{\@tempdimb}{$#1#2$}\@tempdimb=0.8\@tempdimb
                   \ooalign{$#1#2$\crcr%
                         \hfil\rule[-.5mm]{\@tempdimb}{.4pt}\hfil}}
\newdimen\hoogte    \hoogte=14pt    
\newdimen\breedte   \breedte=14pt   
\newdimen\dikte     \dikte=0.5pt    
  {\ifx*#2\let\pointlabel\relax\else\def\pointlabel{#2}\fi
   \refstepcounter{equation}\trivlist
   \item[\hskip\labelsep\theequation.
         \ifx\pointlabel\relax\else\pointlabel\fi]
   \ignorespaces #1
  }{\relax}

\numberwithin{equation}{section}
\swapnumbers
\newtheorem{theorem}[equation]{Theorem}
\newtheorem{lemma}[equation]{Lemma}
\newtheorem{proposition}[equation]{Proposition}
\newtheorem{corollary}[equation]{Corollary}

\theoremstyle{definition}
\newtheorem{definition}[equation]{Definition}
\newtheorem{example}[equation]{Example}
\theoremstyle{remark}
\newtheorem{remark}[equation]{Remark}

\begin{document}
\begin{CJK*}{GBK}{song}
\title[Remarks on shuffles]{Remarks on the shuffle elements of Iwahori--Hecke algebras}
%
\author{Deke Zhao}
\address{\bigskip\hfil\begin{tabular}{l@{}}
          Department of  Mathematics\\
            Beijing Normal University at Zhuhai, Zhuhai, China\\
                         E-mail: \it deke@amss.ac.cn \hfill
          \end{tabular}}
\thanks{Supported partly by Guangdong Basic and Applied Basic Research
 (2023A1515010251) and the National Natural Science Foundation of China (Grant No.  11871107)}
\begin{abstract}Let $H_n(q)$ be the generic Iwahori--Hecke algebra associated to the symmetric group $\mathfrak{S}_n$ of degree $n$ and let $\mathscr{Y}_{n,i}$ be the sum of standard basis of $H_n(q)$ with Coxeter length $i$ for $1\leq i\leq \binom{n}{2}$. We show that $\mathscr{Y}_{n,1}$ and $\mathscr{Y}_{n,2}$ are noncommutative whenever $n\geq 4$, which disproves Doikou's Conjecture on the commutativity of the shuffle elements of $H_n(q)$ in (Nuclear Physics B 1029 (2026): 117532). We also show that the matrix
of the operator of the left multiplication by $\mathscr{Y}_{n,i}$ in $H_n(q)$ is $q$-symmetric for all $i$.
\end{abstract}

\makeatletter
\@namedef{subjclassname@2020}{\textup{2020} Mathematics Subject Classification}
\makeatother
\subjclass[2020]{}

\maketitle



\setcounter{section}{0}

\section{Introduction}
Let $n$ be a positive integer and $q$ an indeterminate. The (generic) Iwahori--Hecke algebra $H_n(q)$ associated to the $n$-th symmetric group $\mathfrak{S}_n$ is the unitary associative algebra over $\mathbb{C}(q)$, the field of rational functions in $q$, generated by $T_1, \ldots, T_{n-1}$ with relations
 \begin{align*}&T_i^2=(q-1)T_i+q &&\text{ for }1\leq i\leq n-1,\\
&T_iT_j=T_jT_i &&\text{ for }|i-j|\geq 2,\\
 &T_iT_{i+1}T_i=T_{i+1}T_{i}T_{i+1} &&\text{ for }1\leq i\leq n-2.\end{align*}
Let $w\in \mathfrak{S}_n$ and let $s_{i_1}s_{i_2}\cdots s_{i_k}$ be a reduced expression for $w$. Then $T_{w}:=T_{i_1}T_{i_2}\cdots T_{i_k}$ is independent of the choice of reduced expression and  $\{T_{w}|w\in \mathfrak{S}_n\}$ is $\mathbb{C}(q)$-linear basis of $H_n(q)$, which is called the standard basis of $H_n(q)$.

\begin{definition}\label{Def:Shuffle}
 For $z\in \mathbb{C}(q)$ and $i=1,\ldots, n-1$, set
\begin{eqnarray*}
\mathscr{X}_i(z)&:=&zT_i+z^2T_{i-1}T_{i}+\cdots+z^iT_1T_2\cdots T_i.
\end{eqnarray*}
The  \textit{shuffle element}  $\mathscr{S}_n$ for $H_n(q)$ is defined as follows: \begin{eqnarray*}
\mathscr{S}_n(z)&:=&\mathscr{X}_{n-1}(z)\mathscr{X}_{n-2}(z)\cdots\mathscr{X}_{1}(z).
\end{eqnarray*}\end{definition}

Now let $\ell$ be the  length function of $\mathfrak{S}_n$. It is easy to see that
\begin{eqnarray}\label{Equ:Shuffle=sum}
\mathscr{S}_n(z)&=&\sum_{k=0}^{\binom{n}{2}}z^k\mathscr{Y}_{n,k},
\end{eqnarray}
where $\mathscr{Y}_{n,k}=\sum_{w\in\mathfrak{S}_n, \ell(w)=k}T_{w}$ and $\mathscr{Y}_{n,0}=1$. For $1\leq i,j\leq \binom{n}{2}$, Doikou   \cite[Conjecture~5.7]{Doikou} conjectures that
 \begin{equation*}
 \mathscr{Y}_{n,i}\mathscr{Y}_{n,j}=\mathscr{Y}_{n,j}\mathscr{Y}_{n,i},
 \end{equation*} or equivalently,  for $z,z'\in \mathbb{C}(q)$,
\begin{equation*}
\mathscr{S}_n(z)\mathscr{S}_n(z')=\mathscr{S}_n(z')\mathscr{S}_n(z).
\end{equation*}
Sometime we also denote by $\mathscr{Y}_{n,i}$ the sum of elements of $\mathfrak{S}_n$ with length $i$ when it is clear from context.

The main result of this paper is following fact, which disproves Doikou's conjecture.

\begin{theorem}\label{Theorem:main}
 Assume that  $n\geq 4$.   Then there exist $1\leq i\neq j\leq \binom{n}{2}$ such that $\mathscr{Y}_{n,i}$ and $\mathscr{Y}_{n,j}$ are noncommutative, or equivalently, for $z,z'\in \mathbb{C}(q)$,
  $\mathscr{S}_n(z)$ and $\mathscr{S}_n(z')$ are noncommutative.
\end{theorem}

Indeed, we prove that $\mathscr{Y}_{n,1}$ and $\mathscr{Y}_{n,2}$ are non-commutative when $n\geq 4$. It is interesting to know when $\mathscr{Y}_{n,i}$ and $\mathscr{Y}_{n,j}$ are commutative for $1\leq i\neq j\leq \binom{n}{2}$. We believe that $\mathscr{Y}_{n,i}$ and $\mathscr{Y}_{n,j}$ are noncommutative for almost all $i\neq j$ when $n$ is large enough (see Remark~\ref{Remark:Almost-all}).

Let $\theta: H_n(q)\longrightarrow H_n(q)$ be the map defined on the standard basis by the rule $\theta(T_w)=T_{w^{-1}}$ for $w\in \mathfrak{S}_n$, extended to $H_n(q)$ by linearity. It is well known that $\theta$ is an anti-automorphism  of $H_n(q)$ (see e.g. \cite{Dipper-James} and see \cite[Lemma~2.3]{Murphy} for a proof). Since $\ell(w^{-1})=\ell(w)$ for $w\in \mathfrak{S}_n$, we have
 \begin{eqnarray*}
\theta(\mathscr{Y}_{n,i})=\mathscr{Y}_{n,i} &\quad\text{ and }\quad& \theta(\mathscr{Y}_{n,i}\mathscr{Y}_{n,j})=
 \theta(\mathscr{Y}_{n,j})\theta(\mathscr{Y}_{n,i})=
 \mathscr{Y}_{n,j}\mathscr{Y}_{n,i}
 \end{eqnarray*}
for all $n\geq 2$ and $1\leq i,j\leq\binom{n}{2}$.
Therefore $\mathscr{Y}_{n,i}$ and $\mathscr{Y}_{n,j}$ are commutative if and only if $\mathscr{Y}_{n,i}\mathscr{Y}_{n,j}$ is $\theta$-invariant, while it is very hard to give a criterion on the $\theta$-invariant of $\mathscr{Y}_{n,i}\mathscr{Y}_{n,j}$ for $1\leq i\neq j\leq\binom{n}{2}$.

Now we can obtain an sufficient condition for the commutativity of $\mathscr{Y}_{n,i}$ and $\mathscr{Y}_{n,j}$, while it is unnecessary (see e.g.  $\mathscr{Y}_{4,1}$ and $\mathscr{Y}_{4,3}$ in Example~\ref{Exam:n=4}).
\begin{theorem}\label{Them:Comm}For $1\leq i\neq j\leq \binom{n}{2}$,  if $\mathscr{Y}_{n,i}\mathscr{Y}_{n,j}$ is a $\mathbb{Z}[q]$-linear combination of $\mathscr{Y}_{n,a}, \mathscr{Y}_{n,a+1}, \ldots, \mathscr{Y}_{n,b}$ with $|i-j|\leq a\leq b\leq\min\{i+j, \binom{n}{2}\}$, then $\mathscr{Y}_{n,i}$ and $\mathscr{Y}_{n,j}$ are commutative.
\end{theorem}

Note that $\mathscr{X}_{n-1}(1)$ is the 1-shuffle element of $H_n(q)$ and Lusztig \cite{Lusztig} established that the spectrum of the operator of the left multiplication by $\mathscr{X}_{n-1}(1)$ in $H_n(q)$ consists of the $q$-numbers
 \begin{equation*}
[i]_q=1+q+\cdots+q^{i-1},\qquad i=0,1,\ldots,n-2,n.
\end{equation*}
Lusztig's work was inspired by Wallach's identity stated that the operator of the left multiplication by the 1-shuffle element in $\mathbb{C}\mathfrak{S}_n$ is diagonalizable with eigenvalues $0,1,2,\ldots, n-2,n$ (see \cite[Proposition~A.1]{Wallach}). Then Diaconis et al \cite{Diaconis-Fill-Pitman} (see also \cite{Phatarfod}) found that the multiplicity of the eigenvalue $i$ is equal to the number of permutations in $\mathfrak{S}_n$ with $i$ fixed points. Later, Isaev and Ogievetsky considered shuffle elements in the braid group ring $\mathbb{Z}B_n$ and constructed theirs additive and multiplicative analogues in Hecke and Birman--Murakami--Wenzl algebras via the Baxterized elements in \cite{Isaev-O}. In \cite{O-Petrova}, Ogievetsky and Petrova introduced the shuffles in the group algebra $\mathbb{Z}G(m,1,n)$ of the complex reflection group of type $G(m,1,n)$ and determined the spectrum and the multiplicities of the operator of the left multiplication by the shuffle element in $\mathbb{Z}G(m,1,n)$
by adopting Diaconis et al's approach \cite{Diaconis-Fill-Pitman}. A point should be noted that Ogievetsky and Petrova also introduced the shuffle elements in the cyclotomic Hecke algebra and proposed a conjecture on the spectrum of the operator of the left multiplication by the shuffle element in the Hecke algebra of type $B$ (see \cite[Conjecture~]{O-Petrova}), which is a
type $B$ analogue of Lusztig's \textit{loc. cit.}  work. 

Now for $i=1,\ldots,\binom{n}{2}$, let $\mathscr{M}^q_{n,i}=(a_{\alpha,\beta})_{\alpha,\beta\in\mathfrak{S}_n}$ (resp.  $\mathscr{M}_{n,i}$) be the matrix corresponding to the operator of the left multiplication by $\mathscr{Y}_{n,i}$ in $H_n(q)$ (resp. $\mathbb{C}\mathfrak{S}_n)$ with respect to the length-plus lexicographic order on the standard basis. Inspired by aforementioned works, we show that $\mathscr{M}^q_{n,i}$ are $q$-symmetric for all $i$, that is, 
\begin{equation*}
q^{\ell(\alpha)}a_{\alpha,\beta}=q^{\ell(\beta)}a_{\beta,\alpha} \text{ for $\alpha,\beta\in\mathfrak{S}_n$.}
\end{equation*}
In particular, $\mathscr{M}_{n,i}$ is symmetric, and so it is diagonalizable, which shows $\mathscr{M}_{n,i}^q$ also is diagonalizable  for all $i$, by noticing that $H_n(q)$ are $q$-deformation of $\mathbb{C}\mathfrak{S}_n$. Unfortunately, we can not determine completely the eigenvalues of $\mathscr{M}_{n,i}$ and $\mathscr{M}^q_{n,i}$ for all $i$.  Clearly, $\mathscr{Y}_{n,i}$ acts not just on $\mathbb{C}\mathfrak{S}_n$, but on any of its modules, i.e., on any representations of $\mathfrak{S}_n$. Thus we can consider the eigenvalues of the $\mathscr{Y}_{n,i}$-action on the Specht modules, and so the eigenvalues of $\mathscr{M}_{n,i}$ can be determined explicitly by the $\mathscr{Y}_{n,i}$-action on the Specht modules. Furthermore, for $i=1, \ldots, \binom{n}{2}$, the eigenvalues of $\mathscr{M}_{n,i}$ and $\mathscr{M}_{n,\binom{n}{2}-i}$ can be mutually determined (see the end of Section~\ref{Sec:q-symmetric}).

This paper is organized as follows. We begin with a recursion formula  for $\mathscr{Y}_{n,i}$ and some examples in Section~\ref{Sec:Recursion}. Theorem~\ref{Theorem:main} is proved in Section~\ref{Sec:Proof-Theorem} by establishing a recursion formula for the difference of $\mathscr{Y}_{n,1}\mathscr{Y}_{n,2}$  and $\mathscr{Y}_{n,1}\mathscr{Y}_{n,2}$. In Section~\ref{Sec:q-symmetric}, we prove the $q$-symmetry of the matrices corresponding to the operators of the left multiplication by $\mathscr{Y}_{n,2}$ in $H_n(q)$. In the last section we give a representation-interpretation on the eigenvalues of these matrices.

\section{A recursion formula for $\mathscr{Y}_{n,i}$}\label{Sec:Recursion}
Given an integer $i$, we set $\mathscr{Y}_{n,i}=0$ when $i\notin\{0,1,2,\ldots,\binom{n}{2}\}$. Definition~\ref{Def:Shuffle} and Eq.~\eqref{Equ:Shuffle=sum} show that
\begin{eqnarray*}
\mathscr{S}_{n+1}(z)&=&\sum_{i=0}^{\binom{n+1}{2}}z^i\mathscr{Y}_{n+1,i}\\
&=&(zT_n+z^2T_{n-1}T_{n}+\cdots+z^nT_1T_2\cdots T_n)\sum_{i=0}^{\binom{n}{2}}z^i\mathscr{Y}_{n,i}\\
&=&\sum_{i=0}^{\binom{n+1}{2}}z^i\biggl(\mathscr{Y}_{n,i}+
\sum_{j=1}^{n}T_{n+1-j}T_{n-j}\cdots T_n\mathscr{Y}_{n,i-j}
\biggr).
\end{eqnarray*}

Comparing the coefficients of $z^i$, we obtain the following recursion formula for $\mathscr{Y}_{n,i}$.

\begin{lemma}\label{Lemm:Recursion-formula}
For a positive integer $n$, we have
 \begin{equation*}
 \mathscr{Y}_{n+1,i}=\mathscr{Y}_{n,i}+
\sum_{j=1}^{n}T_{n+1-j}T_{n-j}\cdots T_n\mathscr{Y}_{n,i-j}.
 \end{equation*}
\end{lemma}

\begin{example} By Lemma~\ref{Lemm:Recursion-formula}, we have
\begin{eqnarray*}
  \mathscr{Y}_{3,1}&=&\mathscr{Y}_{2,1}+T_2\mathscr{Y}_{2,0}=T_1+T_2,\\
 \mathscr{Y}_{3,2}&=&T_2\mathscr{Y}_{2,1}+T_1T_2\mathscr{Y}_{2,0}=T_2T_1+T_1T_2,\\
 \mathscr{Y}_{3,3}&=&T_1T_2\mathscr{Y}_{2,1}=T_1T_2T_1.
 \end{eqnarray*}
Calculation shows \begin{eqnarray*}
 \mathscr{Y}_{3,1}^2&=&2q\mathscr{Y}_{3,0}+(q-1)\mathscr{Y}_{3,1}+\mathscr{Y}_{3,2},\\
\mathscr{Y}_{3,2}^2&=&2q^2\mathscr{Y}_{3,0}+q(q-1)\mathscr{Y}_{3,1}
+q\mathscr{Y}_{3,2}+ 4(q-1)\mathscr{Y}_{3,3},\\
 \mathscr{Y}_{3,3}^2&=&q^3\mathscr{Y}_{3,0}+q^2(q-1)\mathscr{Y}_{3,1}+
 q(q-1)^2\mathscr{Y}_{3,2}+(q-1)(q^2-q+1)\mathscr{Y}_{3,3},\\
 \mathscr{Y}_{3,1}\mathscr{Y}_{3,2}&=&
 \mathscr{Y}_{3,1}+(q-1)\mathscr{Y}_{3,2}+2\mathscr{Y}_{3,3},\\
 \mathscr{Y}_{3,1}\mathscr{Y}_{3,3}&=&
 q\mathscr{Y}_{3,2}+2(q-1)\mathscr{Y}_{3,3},\\
 \mathscr{Y}_{3,2}\mathscr{Y}_{3,3}&=&
 q^2\mathscr{Y}_{3,1}+2q(q-1)\mathscr{Y}_{3,2}+2(q-1)^2\mathscr{Y}_{3,3}.
 \end{eqnarray*}
 Therefore $\mathscr{Y}_3(q)=\langle \mathscr{Y}_{3,i}|i=0,1,2,3 \rangle$ is a $4$-dimensional commutative subalgebra of $H_3(q)$.
 \end{example}

\begin{example}\label{Exam:n=4}By Lemma~\ref{Lemm:Recursion-formula}, we have
 \begin{eqnarray*}
 \mathscr{Y}_{4,1}&=&\mathscr{Y}_{3,1}+T_3=T_1+T_2+T_3,\\
\mathscr{Y}_{4,2}&=&\mathscr{Y}_{3,2}+T_3\mathscr{Y}_{3,1}+T_2T_3
=T_2T_1+T_1T_2+T_3T_1+T_3T_2+T_2T_3,\\
 \mathscr{Y}_{4,3}&=&\mathscr{Y}_{3,3}+T_3\mathscr{Y}_{3,2}
 +T_2T_3\mathscr{Y}_{3,1}+T_1T_2T_3=T_1T_2T_1+
 T_3T_2T_1+T_3T_1T_2+T_2T_3T_1+T_2T_3T_2+T_1T_2T_3,\\
 \mathscr{Y}_{4,4}&=&T_3\mathscr{Y}_{3,3}+T_2T_3\mathscr{Y}_{3,2}
 +T_1T_2T_3\mathscr{Y}_{3,1}=T_3T_1T_2T_1+T_2T_3T_2T_1+
 T_2T_3T_1T_2+T_1T_2T_3T_1+T_1T_2T_3T_2,\\
 \mathscr{Y}_{4,5}&=&T_2T_3\mathscr{Y}_{3,3}+T_1T_2T_3\mathscr{Y}_{3,2}
 =T_2T_3T_1T_2T_1+T_1T_2T_3T_2T_1+T_1T_2T_3T_2T_1,\\
\mathscr{Y}_{4,6}&=&T_1T_2T_3\mathscr{Y}_{3,3}=T_1T_2T_3T_1T_2T_1.
 \end{eqnarray*}
Calculation  shows that
\begin{eqnarray*}
\mathscr{Y}_{4,1}\mathscr{Y}_{4,2}&=&2q\mathscr{Y}_{4,1}+(q-1)\mathscr{Y}_{4,2}+
\mathscr{Y}_{4,3}+T_1T_2T_1+T_2T_3T_2+T_3T_1T_2+(q-1)T_1T_3;\\
\mathscr{Y}_{4,2}\mathscr{Y}_{4,1}&=&2q\mathscr{Y}_{4,1}+(q-1)\mathscr{Y}_{4,2}+
\mathscr{Y}_{4,3}+T_1T_2T_1+T_2T_3T_2+T_2T_3T_1+(q-1)T_1T_3.
\end{eqnarray*}
Thus
\begin{eqnarray*}
\mathscr{Y}_{4,1}\mathscr{Y}_{4,2}-\mathscr{Y}_{4,2}\mathscr{Y}_{4,1}&=&
T_3T_1T_2-T_2T_1T_3, 
\end{eqnarray*}
which shows $\mathscr{Y}_{4,1}$ and $\mathscr{Y}_{4,2}$ are noncommutative. On the other hand, we have
\begin{eqnarray*}
\mathscr{Y}_{4,1}\mathscr{Y}_{4,3}&=&2q\mathscr{Y}_{4,2}+(q-1)\mathscr{Y}_{4,3}
+\mathscr{Y}_{4,4}+(q-1)(T_1T_2T_1+T_2T_3T_2+T_1T_3T_2)-qT_1T_3,\\
\mathscr{Y}_{4,3}\mathscr{Y}_{4,1}&=&2q\mathscr{Y}_{4,2}+(q-1)\mathscr{Y}_{4,3}
+\mathscr{Y}_{4,4}+(q-1)(T_1T_2T_1+T_2T_3T_2+T_1T_3T_2)-qT_1T_3.
\end{eqnarray*}
Thus $\mathscr{Y}_{4,1}$ and $\mathscr{Y}_{4,3}$ are commutative, while $\mathscr{Y}_{4,1}\mathscr{Y}_{4,3}$ is not a $\mathbb{Z}[q]$-linear combination of $\mathscr{Y}_{4,2}, \mathscr{Y}_{4,3}, \mathscr{Y}_{4,4}$.
Similarly, we  have
\begin{eqnarray*}
\mathscr{Y}_{4,1}\mathscr{Y}_{4,6}&=&q\mathscr{Y}_{4,5}+
3(q-1)\mathscr{Y}_{4,6},\\
\mathscr{Y}_{4,2}\mathscr{Y}_{4,3}&=&q^2\mathscr{Y}_{4,1}
+q(q-1)\mathscr{Y}_{4,2}+q\mathscr{Y}_{4,3}+(q-1)\mathscr{Y}_{4,4}
+\mathscr{Y}_{4,5},\\
\mathscr{Y}_{4,2}\mathscr{Y}_{4,4}&=&q^2\mathscr{Y}_{4,2}
+q(q-1)\mathscr{Y}_{4,3}+q\mathscr{Y}_{4,4}+(q-1)\mathscr{Y}_{4,5}
+\mathscr{Y}_{4,6}.
\end{eqnarray*}
Thus Theorem~\ref{Them:Comm} shows $\mathscr{Y}_{4,1}$ and $\mathscr{Y}_{4,6}$,  $\mathscr{Y}_{4,2}$ and $\mathscr{Y}_{4,3}$, and $\mathscr{Y}_{4,2}$ and $\mathscr{Y}_{4,4}$ are commutative, respectively.
\end{example}

\begin{example}\label{Exam:n=5}
By Lemma~\ref{Lemm:Recursion-formula}, we have
   \begin{eqnarray*}
 \mathscr{Y}_{5,1}&=&\mathscr{Y}_{4,1}+T_4\mathscr{Y}_{4,0}=T_1+T_2+T_3+T_4,\\
\mathscr{Y}_{5,2}&=&\mathscr{Y}_{4,2}+T_4\mathscr{Y}_{4,1}+T_3T_4\mathscr{Y}_{4,0},\\
 \mathscr{Y}_{5,3}&=&\mathscr{Y}_{4,3}+T_4\mathscr{Y}_{4,2}+
 T_3T_4\mathscr{Y}_{4,1}+T_2T_3T_4\mathscr{Y}_{4,0}.\end{eqnarray*}
Then calculation shows 
\begin{eqnarray*}
\mathscr{Y}_{5,1}\mathscr{Y}_{5,2}&=&
(\mathscr{Y}_{4,1}+T_4)(\mathscr{Y}_{4,2}+T_4\mathscr{Y}_{4,1}+T_3T_4)\\
&=&\mathscr{Y}_{4,1}\mathscr{Y}_{4,2}+\mathscr{Y}_{4,1}T_4\mathscr{Y}_{4,1}
+\mathscr{Y}_{4,1}T_3T_4+T_4\mathscr{Y}_{4,2}
+q\mathscr{Y}_{4,1}+(q-1)T_4\mathscr{Y}_{4,1}+T_3T_4T_3\\
&=&\mathscr{Y}_{4,1}\mathscr{Y}_{4,2}+T_2T_4T_3+
\theta\text{-terms},
\end{eqnarray*}
where
\begin{equation*}
\text{$\theta$-terms}=q\mathscr{Y}_{4,1}+3qT_4+
(q-1)(\mathscr{Y}_{4,1}T_4+T_4\mathscr{Y}_{4,1})+T_4\mathscr{Y}_{4,2}
+\mathscr{Y}_{4,2}T_4+2T_3T_4T_3+T_1(T_3T_4+T_4T_3).
\end{equation*}

Now by applying the $\theta$-involution, we obtain
\begin{eqnarray*}
\mathscr{Y}_{5,2}\mathscr{Y}_{5,1}&=&\mathscr{Y}_{4,2}\mathscr{Y}_{4,1}
+T_3T_4T_2+\theta\text{-terms}.
\end{eqnarray*}
As a consequence, we get
\begin{eqnarray*}
  \mathscr{Y}_{5,1}\mathscr{Y}_{5,2}-\mathscr{Y}_{5,2}\mathscr{Y}_{5,1}
    &=&\mathscr{Y}_{4,1}\mathscr{Y}_{4,2}-\mathscr{Y}_{4,2}\mathscr{Y}_{4,1}
  +(T_2T_4T_3-T_3T_4T_2)\\
&=&(T_3T_1T_2-T_2T_1T_3)+(T_4T_2T_3-T_3T_2T_4).
\end{eqnarray*}
\end{example}
\section{Proof of Theorem~\ref{Theorem:main}}\label{Sec:Proof-Theorem}

The following fact is a generalization of Examples~\ref{Exam:n=4} and \ref{Exam:n=5}.

\begin{lemma}\label{Lemm:12-recursion}
Assume that $n$ is not less $3$.  Then
\begin{equation*}
 \mathscr{Y}_{n+1,1}\mathscr{Y}_{n+1,2} =\mathscr{Y}_{n,1}\mathscr{Y}_{n,2}+T_nT_{n-2}T_{n-1}+
 \text{$\theta$-terms},
\end{equation*}
where \begin{eqnarray*}
   \text{$\theta$-terms}&=&\mathscr{Y}_{n,1}T_{n-1}T_n+T_nT_{n-1}\mathscr{Y}_{n,1}+
(q-1)T_n\mathscr{Y}_{n-1,1}
 +\mathscr{Y}_{n,1}T_n\mathscr{Y}_{n,1}+q\mathscr{Y}_{n,1}+T_{n-1}T_nT_{n-1}.
      \end{eqnarray*}
\end{lemma}
\begin{proof}Thanks to Lemma~\ref{Lemm:Recursion-formula},  we have
  \begin{eqnarray*}
 \mathscr{Y}_{n+1,1} &=&\mathscr{Y}_{n,1}+T_n,\\
 \mathscr{Y}_{n+1,2} &=&\mathscr{Y}_{n,2}+T_n\mathscr{Y}_{n,1}+T_{n-1}T_n.
  \end{eqnarray*}
  Thus
  \begin{eqnarray*}
 \mathscr{Y}_{n+1,1}\mathscr{Y}_{n+1,2}&=&\mathscr{Y}_{n,1}\mathscr{Y}_{n,2}
 +\mathscr{Y}_{n,1}T_{n-1}T_n+T_n\mathscr{Y}_{n,2}+(q-1)T_n\mathscr{Y}_{n,1}
 +\mathscr{Y}_{n,1}T_n\mathscr{Y}_{n,1}+q\mathscr{Y}_{n,1}+T_{n-1}T_nT_{n-1}.
 \end{eqnarray*}
Since
\begin{eqnarray*}
 T_n\mathscr{Y}_{n,2}&=&
T_n(\mathscr{Y}_{n-1,2}+T_{n-1}\mathscr{Y}_{n-1,1}+T_{n-2}T_{n-1})\\
&=&T_n\mathscr{Y}_{n-1,2}+T_nT_{n-1}\mathscr{Y}_{n-1,1}+T_{n-2}T_nT_{n-1}\\
&=&T_n\mathscr{Y}_{n-1,2}+T_nT_{n-1}\mathscr{Y}_{n,1}+T_{n-2}T_nT_{n-1}
-(q-1)T_nT_{n-1}-qT_n,
 \end{eqnarray*}
we obtain
\begin{eqnarray*}
 \mathscr{Y}_{n+1,1}\mathscr{Y}_{n+1,2}&=&\mathscr{Y}_{n,1}\mathscr{Y}_{n,2}
 +T_{n-2}T_nT_{n-1}+\text{$\theta$-terms},\end{eqnarray*}
where
\begin{eqnarray*}\text{$\theta$-terms}&=&
\mathscr{Y}_{n,1}T_{n-1}T_n+T_nT_{n-1}\mathscr{Y}_{n,1}+
(q-1)T_n\mathscr{Y}_{n-1,1}
 +\mathscr{Y}_{n,1}T_n\mathscr{Y}_{n,1}+q\mathscr{Y}_{n,1}+T_{n-1}T_nT_{n-1}.
 \end{eqnarray*}
It completes the proof.
\end{proof}

 Now we have the following fact, which determine completely the difference between $\mathscr{Y}_{n,1}\mathscr{Y}_{n,2}$ and $\mathscr{Y}_{n,2}\mathscr{Y}_{n,1}$.

\begin{proposition}\label{Prop:12-diff}If  $n\geq 4$ then
\begin{eqnarray*}
 \mathscr{Y}_{n,1}\mathscr{Y}_{n,2}-\mathscr{Y}_{n,2}\mathscr{Y}_{n,1}&=&
\sum_{i=1}^{n-3}(T_{i+2}T_iT_{i+1}-T_{i+1}T_iT_{i+2}).
\end{eqnarray*}
In particular, $\mathscr{Y}_{n,1}$ and $\mathscr{Y}_{n,2}$ are noncommutative.
\end{proposition}
\begin{proof}We show the equality by applying the induction argument on $n$. For $n=4$, it follows form Example~\ref{Exam:n=4}. Assume that it holds for $n\geq 4$. We show it holds for $n+1$.

Since $\mathscr{Y}_{n+1,2}\mathscr{Y}_{n+1,1}=
\theta(\mathscr{Y}_{n+1,1}\mathscr{Y}_{n+1,2})$ and $\theta$-term is $\theta$-invariant, Lemma~\ref{Lemm:12-recursion} shows
\begin{eqnarray*}
 \mathscr{Y}_{n+1,1}\mathscr{Y}_{n+2,2}-\mathscr{Y}_{n+1,2}\mathscr{Y}_{n+1,1}
 &=&(\mathscr{Y}_{n,1}\mathscr{Y}_{n,2}-\mathscr{Y}_{n,2}\mathscr{Y}_{n,1})
+T_{n-2}T_{n}T_{n-1}-\theta(T_{n-1}T_{n}T_{n-1})\\
 &=&(\mathscr{Y}_{n,1}\mathscr{Y}_{n,2}-\mathscr{Y}_{n,2}\mathscr{Y}_{n,1})
+T_{n-2}T_{n}T_{n-1}-T_{n-1}T_nT_{n-2}\\
&=&\sum_{i=1}^{n-2}(T_{i+2}T_iT_{i+1}-T_{i+1}T_iT_{i+2}).
\end{eqnarray*}
It completes the proof.
\end{proof}

\begin{proof}[Proof of Theorem~\ref{Theorem:main}]The theorem follows directly by applying Proposition~\ref{Prop:12-diff}.
\end{proof}

Motivated by Proposition~\ref{Prop:12-diff}, it is interesting to give a recursion formula for $\mathscr{Y}_{n,i}\mathscr{Y}_{n,j}$, which will provide a necessary and sufficient condition for the commutativity of $\mathscr{Y}_{n,i}$ and $\mathscr{Y}_{n,j}$ for $1\leq i< j\leq \binom{n}{2}$. Unfortunately it is very complicate to do that.  Indeed, by Lemma~\ref{Lemm:Recursion-formula}, we have
\begin{eqnarray*}
\mathscr{Y}_{n+1,i}\mathscr{Y}_{n+1,j}&=&
(\mathscr{Y}_{n,i}+\sum_{a=1}^nT_{n+1-a}\cdots T_n\mathscr{Y}_{n,i-a})(\mathscr{Y}_{n,j}+\sum_{b=1}^nT_{n+1-b}\cdots T_n\mathscr{Y}_{n,j-b})\\
&=&\mathscr{Y}_{n,i}\mathscr{Y}_{n,j}+
\sum_{b=1}^n\mathscr{Y}_{n,i}T_{n+1-b}\cdots T_n\mathscr{Y}_{n,j-b}
+\sum_{a=1}^n\mathscr{Y}_{n,j}T_{n+1-a}\cdots T_n\mathscr{Y}_{n,i-a}\\
&&\qquad
+\sum_{a=1}^n\sum_{b=1}^nT_{n+1-a}\cdots T_n\mathscr{Y}_{n,i-a}T_{n+1-b}\cdots T_n\mathscr{Y}_{n,j-b}.
\end{eqnarray*}
Thus \begin{eqnarray}\label{Equ:ij-diff}
\mathscr{Y}_{n+1,i}\mathscr{Y}_{n+1,j}-
\mathscr{Y}_{n+1,j}\mathscr{Y}_{n+1,i}&=&\mathscr{Y}_{n,i}\mathscr{Y}_{n,j}-
\mathscr{Y}_{n,j}\mathscr{Y}_{n,i}+\text{other terms.}
\end{eqnarray}

The following fact illustrates it is hard to determine explicitly the ``other terms" in Eq.\eqref{Equ:ij-diff}.

\begin{lemma}\label{Lemm:13-recursion}
Assume that $n\geq 3$.  Then
\begin{eqnarray*}
 \mathscr{Y}_{n+1,1}\mathscr{Y}_{n+1,3} &=&\mathscr{Y}_{n,1}\mathscr{Y}_{n,3}+\mathscr{Y}_{n-1,1}
\mathscr{Y}_{n-1,2}T_n+T_nT_{n-2}T_{n-1}
(\mathscr{Y}_{n-3,1}+q-1)+\tilde{\theta},
\end{eqnarray*}
where $\tilde{\theta}$ is a $\theta$-invariants, which can be determined explicitly (see the proof).
\end{lemma}
\begin{proof}Thanks to Lemma~\ref{Lemm:Recursion-formula},  we have
  \begin{eqnarray*}
 \mathscr{Y}_{n+1,1} &=&\mathscr{Y}_{n,1}+T_n,\\
 \mathscr{Y}_{n+1,3} &=&\mathscr{Y}_{n,3}+T_n\mathscr{Y}_{n,2}+
 T_{n-1}T_n\mathscr{Y}_{n,1}+T_{n-2}T_{n-1}T_n.
  \end{eqnarray*}
  Thus
  \begin{eqnarray*}
 \mathscr{Y}_{n\!+\!1,1}\mathscr{Y}_{n\!+\!1,3}\!\!&=&\!
(\mathscr{Y}_{n,1}+T_n)
 (\mathscr{Y}_{n,3}+T_n\mathscr{Y}_{n,2}+
 T_{n-1}T_n\mathscr{Y}_{n,1}+T_{n-2}T_{n-1}T_n)\\
\!\!&=&\!\!\mathscr{Y}_{n,1}\mathscr{Y}_{n,3}
 \!+\!\mathscr{Y}_{n,1}T_{n\!-\!1}T_n\mathscr{Y}_{n,1}\!+\!
 \mathscr{Y}_{n,1}T_n\mathscr{Y}_{n,2}\!+\!(\!q\!-\!1\!)T_n\mathscr{Y}_{n,2}
\!+\!qT_{n\!-\!1}T_n\!+\!
 \mathscr{Y}_{n,1}T_{n\!-\!2}T_{n\!-\!1}T_n\!+\! T_n\mathscr{Y}_{n,3}\!+\!\theta_1,\end{eqnarray*}
 where
\begin{eqnarray*}
\theta_1&=&q\mathscr{Y}_{n,2}+(q-1)T_{n-1}T_nT_{n-1}+
T_{n-1}T_nT_{n-1}\mathscr{Y}_{n-1,1}+T_{n-2}T_{n-1}T_nT_{n-1}.
\end{eqnarray*}

Firstly, by applying Lemma~\ref{Lemm:Recursion-formula} again, we have
 \begin{eqnarray*}
A&:=&\mathscr{Y}_{n,1}T_{n\!-\!1}T_n\mathscr{Y}_{n,1}+
\mathscr{Y}_{n,1}T_n\mathscr{Y}_{n,2}\\
\!\!&=&\!\!\mathscr{Y}_{n,1}T_{n\!-\!1}T_n\mathscr{Y}_{n,1}+\mathscr{Y}_{n,1}T_n (\mathscr{Y}_{n-1,2}+T_{n-1}\mathscr{Y}_{n-1,1}+T_{n-2}T_{n-1})\\
 \!\!&=&\!\mathscr{Y}_{n,1}T_{n\!-\!1}T_n\mathscr{Y}_{n,1}+\mathscr{Y}_{n,1}T_n\mathscr{Y}_{n-1,2}+\mathscr{Y}_{n,1}
T_nT_{n-1}\mathscr{Y}_{n-1,1}+\mathscr{Y}_{n,1}T_{n-2}T_nT_{n-1}\\
\!\!&=&\!\!\mathscr{Y}_{n\!-\!1,1}
\mathscr{Y}_{n\!-\!1,2}T_n\!+\!T_{n\!-\!1}T_n\mathscr{Y}_{n\!-\!1,2}\!+\!
\mathscr{Y}_{n,1}T_{n\!-\!2}T_nT_{n\!-\!1}\!-\!
(q\!-\!1)\mathscr{Y}_{n,1}T_nT_{n\!-\!1}\!-\!q\mathscr{Y}_{n,1}T_n+\theta_A,
\end{eqnarray*}
where $\theta_A=\mathscr{Y}_{n,1}T_{n\!-\!1}T_n\mathscr{Y}_{n,1}+
\mathscr{Y}_{n,1}T_nT_{n\!-\!1}\mathscr{Y}_{n,1}$.

Now let $Z=\!T_{n\!-\!1}T_n\mathscr{Y}_{n\!-\!1,2}\!+\!
\mathscr{Y}_{n,1}T_{n\!-\!2}T_nT_{n\!-\!1}$. Then
\begin{eqnarray*}
 Z&=&\!
T_{n\!-\!1}T_n\mathscr{Y}_{n-2,2}+
T_{n-1}T_nT_{n-3}T_{n-2}+(q-1)T_{n-2}T_nT_{n-1}+q
T_nT_{n-1}+\theta_Z,
\end{eqnarray*}
where $\theta_Z=T_{n-1}T_nT_{n-2}\mathscr{Y}_{n-2,1}+
\mathscr{Y}_{n-2,1}T_{n-2}T_nT_{n-1}$.

Secondly, we let $B:=(q\!-\!1)T_n\mathscr{Y}_{n,2}\!+\!qT_{n\!-\!1}T_n$. Then
\begin{eqnarray*}
B&=&(q-1)\mathscr{Y}_{n,1}T_nT_{n-1}+q\mathscr{Y}_{n,1}T_n+
(q-1)T_nT_{n-1}T_{n-2}+\theta_B,\end{eqnarray*}
where $\theta_B=(q\!-\!1)\mathscr{Y}_{n\!-\!1,2}T_n\!-\!
(q\!-\!1)T_{n\!-\!1}T_nT_{n\!-\!1}-q\mathscr{Y}_{n\!-\!1,1}T_n$.

Thus we get that
\begin{eqnarray*}
A+B&=&\mathscr{Y}_{n-1,1}\mathscr{Y}_{n-1,2}T_n+Z+
(B-(q-1)\mathscr{Y}_{n,1}T_nT_{n-1}-q\mathscr{Y}_{n,1}T_n)+\theta_A\\
&=&\mathscr{Y}_{n-1,1}\mathscr{Y}_{n-1,2}T_n+
T_{n-1}T_n\mathscr{Y}_{n-2,2}+T_{n-1}T_nT_{n-3}
T_{n-2}+(q-1)T_nT_{n-2}T_{n-1}+Y+\theta_{AB},
\end{eqnarray*}
where $Y=(q-1)T_nT_{n-1}T_{n-2}
+qT_nT_{n-1}$ and $\theta_{AB}=\theta_A+\theta_B+\theta_Z$. Further we have \begin{eqnarray*}
\mathscr{Y}_{n,1}T_{n-2}T_{n-1}T_n+Y&=&
\mathscr{Y}_{n-2,1}T_{n\!-\!2}T_{n\!-\!1}T_n+
T_{n-1}T_{n-2}T_{n-1}T_n+\theta_{Y},
\end{eqnarray*}
where $\theta_Y=Y+\theta(Y)$.

Thirdly, let
$C=\mathscr{Y}_{n,1}T_{n-2}T_{n-1}T_n+ T_n\mathscr{Y}_{n,3}$.
since \begin{eqnarray*}
 T_nT_{n\!-\!1}\mathscr{Y}_{n\!-\!1,2}\!&=&\!
T_nT_{n\!-\!1}\mathscr{Y}_{n\!-\!2,2}\!+\!
 T_nT_{n\!-\!1}T_{n\!-\!2}\mathscr{Y}_{n-2,1}\!+\!
T_nT_{n\!-\!1}T_{n\!-\!3}T_{n\!-\!2},
\end{eqnarray*}
we have
\begin{eqnarray*}
C+Y&=&(\mathscr{Y}_{n,1}T_{n\!-\!2}T_{n\!-\!1}T_n\!+Y)+\!T_n\mathscr{Y}_{n-1,3}+
 T_nT_{n-1}\mathscr{Y}_{n-1,2}+T_nT_{n-2}T_{n-1}\mathscr{Y}_{n-1,1}
 +T_nT_{n-3}T_{n-2}T_{n-1}\\
\!\!&=&\!\!T_nT_{n\!-\!1}\mathscr{Y}_{n\!-\!2,2}\!+
\!T_nT_{n\!-\!2}T_{n\!-\!1}
\mathscr{Y}_{n\!-\!2,1}\!+\!\theta_{C},
 \end{eqnarray*}
where $\theta_C=T_n\mathscr{Y}_{n\!-\!1,3}\!+\!\mathscr{Y}_{n-2,1}
T_{n\!-\!2}T_{n\!-\!1}T_n\!+
T_nT_{n\!-\!1}T_{n\!-\!2}\mathscr{Y}_{n-2,1}+
T_{n-2}(T_nT_{n-1}+T_{n-1}T_n)T_{n-2}+\theta_{Y}$.

Finally, let $X=T_nT_{n-2}T_{n-1}\mathscr{Y}_{n-2,1}+
T_{n-1}T_nT_{n-3}T_{n-2}$. Then
\begin{eqnarray*}
&&X=T_nT_{n-2}T_{n-1}\mathscr{Y}_{n-3,1}+\theta_X
\end{eqnarray*}
where $\theta_X=T_{n-2}T_nT_{n-1}T_{n-3}+
T_{n-3}T_{n-1}T_nT_{n-2}$.

Combining above equalities, we obtain
\begin{eqnarray*}
 \mathscr{Y}_{n\!+\!1,1}\mathscr{Y}_{n\!+\!1,3}\!\!&=&\!
\mathscr{Y}_{n,1}\mathscr{Y}_{n,3}
 \!+\!A\!+\!B\!+\!C\!+\!\theta_1\\
&=&\mathscr{Y}_{n,1}\mathscr{Y}_{n,3}+\mathscr{Y}_{n\!-\!1,1}
\mathscr{Y}_{n\!-\!1,2}T_n\!+\!T_{n\!-\!1}T_n\mathscr{Y}_{n\!-\!2,2}
+\!T_{n\!-\!1}T_nT_{n\!-\!3}T_{n\!-\!2}+(q\!-\!1)
T_nT_{n\!-\!2}T_{n\!-\!1}\!+\!Y\!+\!C\!+\!\theta_1\\
\!\!&=&\!\!\mathscr{Y}_{n,1}\mathscr{Y}_{n,3}\!+\!\mathscr{Y}_{n\!-\!1,1}
\mathscr{Y}_{n\!-\!1,2}T_n\!+\!T_nT_{n\!-\!2}T_{n\!-\!1}
(\mathscr{Y}_{n-3,1}+q-1)\!+\!\theta,
\end{eqnarray*}
 where
\begin{eqnarray*}
\tilde{\theta}&=&(T_{n-1}T_n+T_nT_{n-1})\mathscr{Y}_{n-2,2}
+\theta_1+\theta_{AB}+\theta_C+\theta_X,
\end{eqnarray*}
which is $\theta$-invariant.
\end{proof}

Now Lemma~\ref{Lemm:13-recursion} enables us to obtain the following recursion formula, which determine completely the difference between $\mathscr{Y}_{n,1}\mathscr{Y}_{n,3}$ and $\mathscr{Y}_{n,3}\mathscr{Y}_{n,1}$.
\begin{corollary}For $n\geq 4$, we have
\begin{eqnarray*}
\mathscr{Y}_{n\!+\!1,1}\mathscr{Y}_{n\!+\!1,3}\!-\!
\mathscr{Y}_{n\!+\!1,3}\mathscr{Y}_{n\!+\!1,1}&=&
\mathscr{Y}_{n,1}\mathscr{Y}_{n,3}\!-\!
\mathscr{Y}_{n,3}\mathscr{Y}_{n,1}+
(\mathscr{Y}_{n-1,1}\mathscr{Y}_{n-1,3}-
\mathscr{Y}_{n-1,3}\mathscr{Y}_{n-1,1})T_n\\
&&\qquad
+(\mathscr{Y}_{n-3,1}+q-1)
(T_{n-2}T_nT_{n-1}-T_{n-1}T_nT_{n-2}).
\end{eqnarray*}
\end{corollary}

\begin{remark}\label{Remark:Almost-all}Based on the above calculations, it is reasonable to suspect that for almost all $1 \leq i\neq j\leq \binom{n}{2}$, $\mathscr{Y}_{n,i}$ and $\mathscr{Y}_{n,j}$ are noncommutative when $n$ is large enough.  In particular, we believe that if  $\mathscr{Y}_{n,i}$ and $\mathscr{Y}_{n,j}$ are noncommutative, then $\mathscr{Y}_{n+1,i}$ and $\mathscr{Y}_{n+1,j}$ are noncommutative.
\end{remark}


\section{The operators of the left multiplications}\label{Sec:q-symmetric}
Recall that the \textit{length-plus lexicographic order} or \textit{the left tree order} on $\mathfrak{S}_n$ (see \cite[Remark~2.1]{Doikou}, \cite{Lawson}), is defined as follows:
\begin{itemize}
  \item[(1)] $s_1<s_2<\cdots<s_{n-1}$;
  \item[(2)] for $x,y\in \mathfrak{S}_n$, $x<y$ if $\ell(x)<\ell(y)$, or $\ell(x)=\ell(y)$ and $x$ is located to the left of $y$.
\end{itemize}
Naturally, the length-plus lexicographic order on $\mathfrak{S}_n$  induces an order on  he standard basis of $H_n(q)$.

\begin{example}Let $e$ be the identity of $\mathfrak{S}_n$. Then we have
 \begin{eqnarray*}
  & e<s_1<s_2<s_1s_2<s_2s_1<s_1s_2s_1;&\\
    &e<s_1<s_2<s_3<s_1s_2<s_1s_3<s_2s_1<s_2s_3<s_3s_2&\\
  &<s_1s_2s_1<s_1s_2s_3<s_1s_3s_2 <s_2s_1s_3<s_2s_3s_2<s_{3}s_2s_1&\\
  &<s_{1}s_2s_1s_3<s_1s_2s_3s_2
  <s_1s_3s_2s_1<s_2s_1s_3s_2<s_2s_3s_2s_1&\\
  & <s_1s_2s_1s_3s_2<s_1s_2s_3s_2s_1
  <s_2s_1s_3s_2s_1<s_1s_2s_3s_1s_2s_1.&
 \end{eqnarray*}
\end{example}

\begin{example}Let $e$ be the identity of $\mathfrak{S}_n$.
Then we have
 \begin{eqnarray*}
  & T_{e}<T_1<T_2<T_1T_2<T_2T_1<T_1T_2T_1;&\\
   &T_{e}<T_1<T_2<T_3<T_1T_2<T_1T_3<T_2T_1<T_2T_3<T_3T_2&\\
  &\quad<T_1T_2T_1<T_1T_2T_3<T_1T_3T_2 <T_2T_1T_3<T_2T_3T_2<T_{3}T_2T_1&\\
  &\quad<T_{1}T_2T_1T_3<T_1T_2T_3T_2
  <T_1T_3T_2T_1<T_2T_1T_3T_2<T_2T_3T_2T_1&\\
&\quad< T_1T_2T_1T_3T_2<T_1T_2T_2T_2T_1
<T_2T_1T_3T_2T_1<T_1T_2T_3T_1T_2T_1.&
 \end{eqnarray*}
\end{example}

For $i=1,\ldots,\binom{n}{2}$, let $\mathscr{M}^q_{n,i}$ be the matrix corresponding to the operator of the left multiplication by $\mathscr{Y}_{n,i}$ in $H_n(q)$ with respect to the length-plus lexicographic order. It is natural to expect that $\mathscr{M}^q_{n,i}$ has some nice properties. For simplicity, we write $\mathscr{M}_{n,i}$ for $\mathscr{M}^q_{n,i}$ when $q=1$.

\begin{example}\label{Exam:n=3-eigenvalues}
Keeping notations as above, we have
\begin{eqnarray*}
   \mathscr{M}^q_{3,1}=\left(\begin{array}{cccccc}
     0&q&q&0&0&0\\1&q\!-\!1&0&0&1&0\\
     1&0&q\!-\!1&q&0&0\\0&0&1&q\!-\!1&0&q\\
     0&1&0&0&q\!-\!1&q\\0&0&0&1&1&2(q\!-\!1)
   \end{array}
   \right),&\quad&
   \mathscr{M}^q_{3,2}=\left(\begin{array}{cccccc}
     0&0&0&q^2&q^2&0\\0&0&q&0&q(q\!-\!1)&q^2\\
     0&q&0&q(q\!-\!1)&0&q^2\\1&0&q\!-\!1&0&q&2q(q\!-\!1)\\
     1&q\!-\!1&0&q&0&2q(q\!-\!1)\\0&1&1&2(q\!-\!1)&2(q\!-\!1)&2(q\!-\!1)^2
   \end{array}
   \right),\end{eqnarray*}
\begin{eqnarray*}
  &&\mathscr{M}^q_{3,3}=\left(\begin{array}{cccccc}
     0&0&0&0&0&q^3\\0&0&0&q^2&0&q^2(q\!-\!1)\\
     0&0&0&0&q^2&q^2(q-1)\\0&q&0&q(q\!-\!1)&q(q\!-\!1)&q(q\!-\!1)^2\\
     0&0&q&q(q\!-\!1)&q(q\!-\!1)&q(q\!-\!1)^2\\
1&q\!-\!1&q\!-\!1&(q\!-\!1)^2&(q\!-\!1)^2&(q\!-\!1)^3+q(q\!-\!1)
   \end{array}
   \right).
 \end{eqnarray*}
In particular,
 \begin{eqnarray*}
  \mathscr{M}_{3,1}=\left(\begin{array}{cccccc}
     0&1&1&0&0&0\\1&0&0&0&1&0\\
     1&0&0&1&0&0\\0&0&1&0&0&1\\
     0&1&0&0&0&1\\0&0&0&1&1&0
   \end{array}
   \right), &\qquad
   \mathscr{M}_{3,2}=\left(\begin{array}{cccccc}
     0&0&0&1&1&0\\0&0&1&0&0&1\\
     0&1&0&0&0&1\\1&0&0&0&1&0\\
     1&0&0&1&0&0\\0&1&1&0&0&0
   \end{array}
   \right),&\qquad\mathscr{M}_{3,3}=\left(\begin{array}{cccccc}
     0&0&0&0&0&1\\0&0&0&1&0&0\\
     0&0&0&0&1&0\\0&1&0&0&0&0\\
     0&0&1&0&0&0\\1&0&0&0&0&0
   \end{array}
   \right).
 \end{eqnarray*}
Then $\mathscr{M}_{3,1}$ is the incidence matrix of cycle graph $C_6$, and its eigenvalues are $2\cos \frac{2k\pi}{6}$, $k=0,1,2,3,4,5$; $\mathscr{M}_{3,2}$ is the incidence matrix of the direct sum of two 3-cycle graphs and its
  eigenvalues are  $2$ and $-1$ with multiplicities $2$ and $4$, respectively; $\mathscr{M}_{3,3}$ is the direct sum of $\left(\begin{array}{cc}0&1\\1&0
  \end{array}\right)$ with multiplicity $3$, and its eigenvalues  are $1,-1$ with multiplicity $3$.
 \end{example}

\begin{example}Keeping notations as above, we have
  \begin{eqnarray*}
\tiny\mathscr{M}_{4,1}\!=\!\left(\!\begin{array}{cccccccccccccccccccccccc}
0&1&1&1&0&0&0&0&0&0&0&0&0&0&0&0&0&0&0&0&0&0&0&0\\
1&0&0&0&0&1&1&0&0&0&0&0&0&0&0&0&0&0&0&0&0&0&0&0\\
1&0&0&0&1&0&0&0&1&0&0&0&0&0&0&0&0&0&0&0&0&0&0&0\\
1&0&0&0&0&1&0&1&0&0&0&0&0&0&0&0&0&0&0&0&0&0&0&0\\
0&0&1&0&0&0&0&0&0&1&0&1&0&0&0&0&0&0&0&0&0&0&0&0\\
0&1&0&1&0&0&0&0&0&0&0&0&1&0&0&0&0&0&0&0&0&0&0&0\\
0&1&0&0&0&0&0&0&0&1&0&0&0&0&1&0&0&0&0&0&0&0&0&0\\
0&0&0&1&0&0&0&0&0&0&1&0&0&1&0&0&0&0&0&0&0&0&0&0\\
0&0&1&0&0&0&0&0&0&0&0&1&0&1&0&0&0&0&0&0&0&0&0&0\\
0&0&0&0&1&0&1&0&0&0&0&0&0&0&0&0&0&1&0&0&0&0&0&0\\
0&0&0&0&0&0&0&1&0&0&0&0&0&0&0&1&1&0&0&0&0&0&0&0\\
0&0&0&0&1&0&0&0&1&0&0&0&0&0&0&0&0&0&1&0&0&0&0&0\\
0&0&0&0&0&1&0&0&0&0&0&0&0&0&0&1&0&0&0&1&0&0&0&0\\
0&0&0&0&0&0&0&1&1&0&0&0&0&0&0&0&1&0&0&0&0&0&0&0\\
0&0&0&0&0&0&1&0&0&0&0&0&0&0&0&0&0&1&0&1&0&0&0&0\\
0&0&0&0&0&0&0&0&0&0&1&0&1&0&0&0&0&0&0&0&0&1&0&0\\
0&0&0&0&0&0&0&0&0&0&1&0&0&1&0&0&0&0&0&0&1&0&0&0\\
0&0&0&0&0&0&0&0&0&1&0&0&0&0&1&0&0&0&0&0&0&0&1&0\\
0&0&0&0&0&0&0&0&0&0&0&1&0&0&0&0&0&0&0&0&1&0&1&0\\
0&0&0&0&0&0&0&0&0&0&0&0&1&0&1&0&0&0&0&0&0&1&0&0\\
0&0&0&0&0&0&0&0&0&0&0&0&0&0&0&0&1&0&1&0&0&0&0&1\\
0&0&0&0&0&0&0&0&0&0&0&0&0&0&0&1&0&0&0&1&0&0&0&1\\
0&0&0&0&0&0&0&0&0&0&0&0&0&0&0&0&0&1&1&0&0&0&0&1\\
0&0&0&0&0&0&0&0&0&0&0&0&0&0&0&0&0&0&0&0&1&1&1&0
\end{array}\!\right),&&
\tiny\mathscr{M}_{4,2}\!=\!\left(\!\begin{array}{cccccccccccccccccccccccc}
0&0&0&0&1&1&1&1&1&0&0&0&0&0&0&0&0&0&0&0&0&0&0&0\\
0&0&1&1&0&0&0&0&0&1&0&0&1&0&1&0&0&0&0&0&0&0&0&0\\
0&1&0&1&0&0&0&0&0&1&0&1&0&1&0&0&0&0&0&0&0&0&0&0\\
0&1&1&0&0&0&0&0&0&0&1&0&1&1&0&0&0&0&0&0&0&0&0&0\\
1&0&0&0&0&0&1&0&1&0&0&0&0&0&0&0&0&1&1&0&0&0&0&0\\
1&0&0&0&0&0&1&1&0&0&0&0&0&0&0&1&0&0&0&1&0&0&0&0\\
1&0&0&0&1&1&0&0&0&0&0&0&0&0&0&0&0&1&0&1&0&0&0&0\\
1&0&0&0&0&1&0&0&1&0&0&0&0&0&0&1&1&0&0&0&0&0&0&0\\
1&0&0&0&1&0&0&1&0&0&0&0&0&0&0&0&1&0&1&0&0&0&0&0\\
0&1&1&0&0&0&0&0&0&0&0&1&0&0&1&0&0&0&0&0&0&0&1&0\\
0&0&0&1&0&0&0&0&0&0&0&0&1&1&0&0&0&0&0&0&1&1&0&0\\
0&0&1&0&0&0&0&0&0&1&0&0&0&1&0&0&0&0&0&0&1&0&1&0\\
0&1&0&1&0&0&0&0&0&0&1&0&0&0&1&0&0&0&0&0&0&1&0&0\\
0&0&1&1&0&0&0&0&0&0&1&1&0&0&0&0&0&0&0&0&1&0&0&0\\
0&1&0&0&0&0&0&0&0&1&0&0&1&0&0&0&0&0&0&0&0&1&1&0\\
0&0&0&0&0&1&0&1&0&0&0&0&0&0&0&0&1&0&0&1&0&0&0&1\\
0&0&0&0&0&0&0&1&1&0&0&0&0&0&0&1&0&0&1&0&0&0&0&1\\
0&0&0&0&1&0&1&0&0&0&0&0&0&0&0&0&0&0&1&1&0&0&0&1\\
0&0&0&0&1&0&0&0&1&0&0&0&0&0&0&0&1&1&0&0&0&0&0&1\\
0&0&0&0&0&1&1&0&0&0&0&0&0&0&0&1&0&1&0&0&0&0&0&1\\
0&0&0&0&0&0&0&0&0&0&1&1&0&1&0&0&0&0&0&0&0&1&1&0\\
0&0&0&0&0&0&0&0&0&0&1&0&1&0&1&0&0&0&0&0&1&0&1&0\\
0&0&0&0&0&0&0&0&0&1&0&1&0&0&1&0&0&0&0&0&1&1&0&0\\
0&0&0&0&0&0&0&0&0&0&0&0&0&0&0&1&1&1&1&1&0&0&0&0
\end{array}\!\right)
 \end{eqnarray*}
Then $\mathscr{M}_{4,1}$ is the incidence matrix of the truncated octahedral graph, which has eigenvalues $\pm3$ (multiplicity 1), $\pm\sqrt{3}$ (multiplicity $2$), $\pm1$ (multiplicity 3),  $\pm(1+\sqrt{2})$ (multiplicity 3), $\pm(\sqrt{2}-1)$(multiplicity 3); $\mathscr{M}_{4,2}$ is the incidence matrix of the disjoint of two Cayley graphs $P_{12}$, which has eigenvalues 
$5$ (multiplicity 2), $\pm\sqrt{5}$ (multiplicity 6), $-1$ (multiplicity 10).
\begin{eqnarray*}
\tiny\mathscr{M}_{4,3}\!=\!\left(\!\begin{array}{cccccccccccccccccccccccc}
0&0&0&0&0&0&0&0&0&1&1&1&1&1&1&0&0&0&0&0&0&0&0&0\\
0&0&0&0&1&0&0&1&1&0&0&0&0&0&0&1&0&1&0&1&0&0&0&0\\
0&0&0&0&0&1&1&1&0&0&0&0&0&0&0&0&1&1&1&0&0&0&0&0\\
0&0&0&0&1&0&1&0&1&0&0&0&0&0&0&1&1&0&0&1&0&0&0&0\\
0&1&0&1&0&0&0&0&0&0&0&0&0&1&1&0&0&0&0&0&1&0&1&0\\
0&0&1&0&0&0&0&0&0&1&1&0&0&1&1&0&0&0&0&0&0&1&0&0\\
0&0&1&1&0&0&0&0&0&0&0&1&1&0&0&0&0&0&0&0&0&1&1&0\\
0&1&1&0&0&0&0&0&0&0&0&1&1&0&0&0&0&0&0&0&1&1&0&0\\
0&1&0&1&0&0&0&0&0&1&1&0&0&0&0&0&0&0&0&0&1&0&1&0\\
1&0&0&0&0&1&0&0&1&0&0&0&0&0&0&0&0&0&1&1&0&0&0&1\\
1&0&0&0&0&1&0&0&1&0&0&0&0&0&0&0&0&0&1&1&0&0&0&1\\
1&0&0&0&0&0&1&1&0&0&0&0&0&0&0&0&1&1&0&0&0&0&0&1\\
1&0&0&0&0&0&1&1&0&0&0&0&0&0&0&0&1&1&0&0&0&0&0&1\\
1&0&0&0&1&1&0&0&0&0&0&0&0&0&0&1&0&0&1&0&0&0&0&1\\
1&0&0&0&1&1&0&0&0&0&0&0&0&0&0&1&0&0&1&0&0&0&0&1\\
0&1&0&1&0&0&0&0&0&0&0&0&0&1&1&0&0&0&0&0&1&0&1&0\\
0&0&1&1&0&0&0&0&0&0&0&1&1&0&0&0&0&0&0&0&0&1&1&0\\
0&1&1&0&0&0&0&0&0&0&0&1&1&0&0&0&0&0&0&0&1&1&0&0\\
0&0&1&0&0&0&0&0&0&1&1&0&0&1&1&0&0&0&0&0&0&1&0&0\\
0&1&0&1&0&0&0&0&0&1&1&0&0&0&0&0&0&0&0&0&1&0&1&0\\
0&0&0&0&1&0&0&1&1&0&0&0&0&0&0&1&0&1&0&1&0&0&0&0\\
0&0&0&0&0&1&1&1&0&0&0&0&0&0&0&0&1&1&1&0&0&0&0&0\\
0&0&0&0&1&0&1&0&1&0&0&0&0&0&0&1&1&0&0&1&0&0&0&0\\
0&0&0&0&0&0&0&0&0&1&1&1&1&1&1&0&0&0&0&0&0&0&0&0
\end{array}\!\right),&&
\tiny\mathscr{M}_{4,4}\!=\!\!\left(\!\begin{array}{cccccccccccccccccccccccc}
0&0&0&0&0&0&0&0&0&0&0&0&0&0&0&1&1&1&1&1&0&0&0&0\\
0&0&0&0&0&0&0&0&0&0&1&1&0&1&0&0&0&0&0&0&0&1&1&0\\
0&0&0&0&0&0&0&0&0&0&1&0&1&0&1&0&0&0&0&0&1&0&1&0\\
0&0&0&0&0&0&0&0&0&1&0&1&0&0&1&0&0&0&0&0&1&1&0&0\\
0&0&0&0&0&1&0&1&0&0&0&0&0&0&0&0&1&0&0&1&0&0&0&1\\
0&0&0&0&1&0&0&0&1&0&0&0&0&0&0&0&1&1&0&0&0&0&0&1\\
0&0&0&0&0&0&0&1&1&0&0&0&0&0&0&1&0&0&1&0&0&0&0&1\\
0&0&0&0&1&0&1&0&0&0&0&0&0&0&0&0&0&0&1&1&0&0&0&1\\
0&0&0&0&0&1&1&0&0&0&0&0&0&0&0&1&0&1&0&0&0&0&0&1\\
0&0&0&1&0&0&0&0&0&0&0&0&1&1&0&0&0&0&0&0&1&1&0&0\\
0&1&1&0&0&0&0&0&0&0&0&1&0&0&1&0&0&0&0&0&0&0&1&0\\
0&1&0&1&0&0&0&0&0&0&1&0&0&0&1&0&0&0&0&0&0&1&0&0\\
0&0&1&0&0&0&0&0&0&1&0&0&0&1&0&0&0&0&0&0&1&0&1&0\\
0&1&0&0&0&0&0&0&0&1&0&0&1&0&0&0&0&0&0&0&0&1&1&0\\
0&0&1&1&0&0&0&0&0&0&1&1&0&0&0&0&0&0&0&0&1&0&0&0\\
1&0&0&0&0&0&1&0&1&0&0&0&0&0&0&0&0&1&1&0&0&0&0&0\\
1&0&0&0&1&1&0&0&0&0&0&0&0&0&0&0&0&1&0&1&0&0&0&0\\
1&0&0&0&0&1&0&0&1&0&0&0&0&0&0&1&1&0&0&0&0&0&0&0\\
1&0&0&0&0&0&1&1&0&0&0&0&0&0&0&1&0&0&0&1&0&0&0&0\\
1&0&0&0&1&0&0&1&0&0&0&0&0&0&0&0&1&0&1&0&0&0&0&0\\
0&0&1&1&0&0&0&0&0&1&0&0&1&0&1&0&0&0&0&0&0&0&0&0\\
0&1&0&1&0&0&0&0&0&1&0&1&0&1&0&0&0&0&0&0&0&0&0&0\\
0&1&1&0&0&0&0&0&0&0&1&0&1&1&0&0&0&0&0&0&0&0&0&0\\
0&0&0&0&1&1&1&1&1&0&0&0&0&0&0&0&0&0&0&0&0&0&0&0
\end{array}\!\right)
\end{eqnarray*}
Thus $\mathscr{M}_{4,3}$ has eigenvalues $\pm6$ (multiplicity 1), $\pm4$ (multiplicity 2), $\pm2$ (multiplicity 1), $0$ (multiplicity 16); $\mathscr{M}_{4,4}$ is the adjacent matrix of the union of two 5-regular graphes with 24-vertices, which has eigenvalues $5$ (multiplicity 2), $\pm\sqrt{5}$ (multiplicity 6), $-1$ (multiplicity 10).
\begin{eqnarray*}
\tiny\mathscr{M}_{4,5}\!=\!\left(\!\begin{array}{cccccccccccccccccccccccc}
0&0&0&0&0&0&0&0&0&0&0&0&0&0&0&0&0&0&0&0&1&1&1&0\\
0&0&0&0&0&0&0&0&0&0&0&0&0&0&0&0&1&0&1&0&0&0&0&1\\
0&0&0&0&0&0&0&0&0&0&0&0&0&0&0&1&0&0&0&1&0&0&0&1\\
0&0&0&0&0&0&0&0&0&0&0&0&0&0&0&0&0&1&1&0&0&0&0&1\\
0&0&0&0&0&0&0&0&0&0&1&0&1&0&0&0&0&0&0&0&0&1&0&0\\
0&0&0&0&0&0&0&0&0&0&0&1&0&0&0&0&0&0&0&0&1&0&1&0\\
0&0&0&0&0&0&0&0&0&0&1&0&0&1&0&0&0&0&0&0&1&0&0&0\\
0&0&0&0&0&0&0&0&0&1&0&0&0&0&1&0&0&0&0&0&0&0&1&0\\
0&0&0&0&0&0&0&0&0&0&0&0&1&0&1&0&0&0&0&0&0&1&0&0\\
0&0&0&0&0&0&0&1&0&0&0&0&0&0&0&1&1&0&0&0&0&0&0&0\\
0&0&0&0&1&0&1&0&0&0&0&0&0&0&0&0&0&1&0&0&0&0&0&0\\
0&0&0&0&0&1&0&0&0&0&0&0&0&0&0&1&0&0&0&1&0&0&0&0\\
0&0&0&0&1&0&0&0&1&0&0&0&0&0&0&0&0&0&1&0&0&0&0&0\\
0&0&0&0&0&0&1&0&0&0&0&0&0&0&0&0&0&1&0&1&0&0&0&0\\
0&0&0&0&0&0&0&1&1&0&0&0&0&0&0&0&1&0&0&0&0&0&0&0\\
0&0&1&0&0&0&0&0&0&1&0&1&0&0&0&0&0&0&0&0&0&0&0&0\\
0&1&0&0&0&0&0&0&0&1&0&0&0&0&1&0&0&0&0&0&0&0&0&0\\
0&0&0&1&0&0&0&0&0&0&1&0&0&1&0&0&0&0&0&0&0&0&0&0\\
0&1&0&1&0&0&0&0&0&0&0&0&1&0&0&0&0&0&0&0&0&0&0&0\\
0&0&1&0&0&0&0&0&0&0&0&1&0&1&0&0&0&0&0&0&0&0&0&0\\
1&0&0&0&0&1&1&0&0&0&0&0&0&0&0&0&0&0&0&0&0&0&0&0\\
1&0&0&0&1&0&0&0&1&0&0&0&0&0&0&0&0&0&0&0&0&0&0&0\\
1&0&0&0&0&1&0&1&0&0&0&0&0&0&0&0&0&0&0&0&0&0&0&0\\
0&1&1&1&0&0&0&0&0&0&0&0&0&0&0&0&0&0&0&0&0&0&0&0
\end{array}\!\right),&&\tiny\mathscr{M}_{4,6}\!=\!
\left(\!\begin{array}{cccccccccccccccccccccccc}
0&0&0&0&0&0&0&0&0&0&0&0&0&0&0&0&0&0&0&0&0&0&0&1\\
0&0&0&0&0&0&0&0&0&0&0&0&0&0&0&0&0&0&0&0&1&0&0&0\\
0&0&0&0&0&0&0&0&0&0&0&0&0&0&0&0&0&0&0&0&0&1&0&0\\
0&0&0&0&0&0&0&0&0&0&0&0&0&0&0&0&0&0&0&0&0&0&1&0\\
0&0&0&0&0&0&0&0&0&0&0&0&0&0&0&1&0&0&0&0&0&0&0&0\\
0&0&0&0&0&0&0&0&0&0&0&0&0&0&0&0&0&0&1&0&0&0&0&0\\
0&0&0&0&0&0&0&0&0&0&0&0&0&0&0&0&1&0&0&0&0&0&0&0\\
0&0&0&0&0&0&0&0&0&0&0&0&0&0&0&0&0&1&0&0&0&0&0&0\\
0&0&0&0&0&0&0&0&0&0&0&0&0&0&0&0&0&0&0&1&0&0&0&0\\
0&0&0&0&0&0&0&0&0&0&1&0&0&0&0&0&0&0&0&0&0&0&0&0\\
0&0&0&0&0&0&0&0&0&1&0&0&0&0&0&0&0&0&0&0&0&0&0&0\\
0&0&0&0&0&0&0&0&0&0&0&0&1&0&0&0&0&0&0&0&0&0&0&0\\
0&0&0&0&0&0&0&0&0&0&0&1&0&0&0&0&0&0&0&0&0&0&0&0\\
0&0&0&0&0&0&0&0&0&0&0&0&0&0&1&0&0&0&0&0&0&0&0&0\\
0&0&0&0&0&0&0&0&0&0&0&0&0&1&0&0&0&0&0&0&0&0&0&0\\
0&0&0&0&1&0&0&0&0&0&0&0&0&0&0&0&0&0&0&0&0&0&0&0\\
0&0&0&0&0&0&1&0&0&0&0&0&0&0&0&0&0&0&0&0&0&0&0&0\\
0&0&0&0&0&0&0&1&0&0&0&0&0&0&0&0&0&0&0&0&0&0&0&0\\
0&0&0&0&0&1&0&0&0&0&0&0&0&0&0&0&0&0&0&0&0&0&0&0\\
0&0&0&0&0&0&0&0&1&0&0&0&0&0&0&0&0&0&0&0&0&0&0&0\\
0&1&0&0&0&0&0&0&0&0&0&0&0&0&0&0&0&0&0&0&0&0&0&0\\
0&0&1&0&0&0&0&0&0&0&0&0&0&0&0&0&0&0&0&0&0&0&0&0\\
0&0&0&1&0&0&0&0&0&0&0&0&0&0&0&0&0&0&0&0&0&0&0&0\\
1&0&0&0&0&0&0&0&0&0&0&0&0&0&0&0&0&0&0&0&0&0&0&0
\end{array}\!\right)
\end{eqnarray*}
Thus $\mathscr{M}_{4,5}$ is the adjacent matrix of the connected 3-regular graph with 24 vertices, which has eigenvalues $\pm3$ (multiplicity 1), $\pm\sqrt{3}$ (multiplicity $2$), $\pm1$ (multiplicity 3),  $\pm(1+\sqrt{2})$ (multiplicity 3), $\pm(\sqrt{2}-1)$ (multiplicity 3); $\mathscr{M}_{4,6}$ is the adjacent matrix of the 1-regular graph with 24 vertices, which has eigenvalues $\pm1$ (multiplicity 12).
\end{example}

Note that $\mathscr{M}_{3,i}$ ($i=1,2,3$) and $\mathscr{M}_{4,i}$ ($i=1,\ldots, 6)$ are symmetric matrices. Indeed, we can show that $\mathscr{M}^q_{n,i}$ is a $q$-symmetric matrix for $1\leq i\leq \binom{n}{2}$. To this end, we need more notations and facts.
Let $\tau$ be the canonical symmetrizing form on $H_n(q)$, that is, for $x,y\in \mathfrak{S}_n$,
\begin{equation}\label{Equ:Canonical-symm-form}
  \tau(T_xT_{y})=\left\{\begin{array}{ll}
q^{\ell(x)}, & \hbox{if $y=x^{-1}$;}\\0,&\hbox{otherwise.}
 \end{array}\right.
 \end{equation}
Then $\tau$ is a symmetric trace on  $H_n(q)$ when $q$ is invertible and $\tau(abc)=\tau(ac\theta(b))$ for $a,b,c\in H_n(q)$ (see  \cite[Lemma~2.2]{Dipper-James}). In particular, $\tau$ is $\theta$-invariant, that is, $\tau(\theta(x))=\tau(x)$ for $x\in H_n(q)$.

\begin{theorem}\label{Them:Matrix-symm}Assume that $\mathscr{M}^q_{n,i}=(a_{\alpha,\beta})_{\alpha,\beta\in\mathfrak{S}_n}$ for $i=1, \ldots, \binom{n}{2}$. Then $\mathscr{M}^q_{n,i}$ is $q$-symmetric, that is,
 \begin{equation*}
  q^{\ell(\alpha)}a_{\alpha,\beta}=q^{\ell(\beta)}a_{\beta,\alpha}
 \end{equation*}
for $\alpha,\beta\in\mathfrak{S}_n$.
In particular, $\mathscr{M}_{n,i}$ is a symmetric matrix.
\end{theorem}

\begin{proof} Given $\alpha,\beta\in \mathfrak{S}_n$, we have
\begin{eqnarray*}
\mathscr{Y}_{n,i}T_{\alpha}T_{\beta^{-1}}&=&
\sum_{y\in\mathfrak{S}_n}a_{y,\alpha}T_yT_{\beta^{-1}}.
\end{eqnarray*}
Thus
\begin{eqnarray*}
\tau(\mathscr{Y}_{n,i}T_{\alpha}T_{\beta^{-1}})&=&
a_{\beta,\alpha}\tau(T_{\beta}T_{\beta^{-1}})=q^{\ell(\beta)}a_{\beta,\alpha}.
\end{eqnarray*}
Similarly, we have
\begin{eqnarray*}
\tau(\mathscr{Y}_{n,i}T_{\beta}T_{\alpha^{-1}})&=&
a_{\alpha,\beta}\tau(T_{\alpha}T_{\alpha^{-1}})=
q^{\ell(\alpha)}a_{\alpha,\beta}.
\end{eqnarray*}
Now applying the $\theta$-invariance of $\tau$, we obtian
\begin{eqnarray*}
&&q^{\ell(\beta)}a_{\beta,\alpha}=\tau(\mathscr{Y}_{n,i}T_{\alpha}T_{\beta^{-1}})
=\tau(\theta(\mathscr{Y}_{n,i}T_{\alpha}T_{\beta^{-1}}))=
\tau(T_{\beta}T_{\alpha^{-1}}\mathscr{Y}_{n,i})=
\tau(\mathscr{Y}_{n,i}T_{\beta}T_{\alpha^{-1}})=q^{\ell(\alpha)}a_{\alpha,\beta}.
\end{eqnarray*}
It completes the proof.
\end{proof}

\begin{remark}Theorem~\ref{Them:Matrix-symm} shows $\mathscr{M}_{n,i}$ are symmetric matrices for all $i=1, \ldots, \binom{n}{2}$. It would be interesting to determine explicitly the  eigenvalues and the multiplicity of these symmetric matrices. Let $y_{n,i}$ is the number of elements of $\mathfrak{S}_n$ with length $i$ for $i=1,\ldots,\binom{n}{2}$, that is,  $y_{n,i}$ is the Mahonian number, and the Mahonian distribution $y_{n,0},y_{n,1},\ldots, y_{n,\binom{n}{2}}$ is palindromic, unimodal and log-concave. Then it is easy to see that $\mathscr{M}_{n,i}$ has eigenvalues $y_{n,i}$ and $-y_{n,i}$ when $i$ is odd and has eigenvalue $y_{n,i}$ with multiplicity 2 when $i$ is even.
 \end{remark}
 
 The end of this section is devoted to present a decomposition of $\mathscr{M}_{n,i}$. 
 
 Now let $w_0$ be the (unique) element of $\mathfrak{S}_n$ with length $\binom{n}{2}$. Then $w\mapsto w_0w$ is an involution on $\mathfrak{S}_n$, which induces a bijective between $Y_{n,i}$ and $Y_{n,\binom{n}{2}-i}$ for $i=0, 1,\ldots, \binom{n}{2}$.  So there is a order ``$\prec$" on the elements of $\mathfrak{S}_n$ such that the matrix, which is also denoted as $\mathscr{M}_{n,\binom{n}{2}}$, corresponding to the operator of the left multiplication by $\mathscr{Y}_{n,\binom{n}{2}}$ in $\mathfrak{S}_n$ with respect to $\prec$ is the direct sum of  $\left(\begin{array}{cc}0&1\\1&0
  \end{array}\right)$ with multiplicity $\frac{n!}{2}$. Thus there exists a orthogonal matrix $\boldsymbol{U}$ such that
  \begin{equation*}
  \boldsymbol{U}^{\mathrm{T}}
  \mathscr{M}_{n,\binom{n}{2}}\boldsymbol{U}=\mathrm{diag}
  (\mathrm{Id}_{\frac{n!}{2}},-\mathrm{Id}_{\frac{n!}{2}}).
  \end{equation*}
Since $\mathscr{Y}_{n,\binom{n}{2}-i}=w_{0}\mathscr{Y}_{n,i}
=\mathscr{Y}_{n,i}w_0$, we have 
\begin{equation*}
\mathscr{M}_{n,\binom{n}{2}-i}=\mathscr{M}_{n,\binom{n}{2}}\mathscr{M}_{n,i}
=\mathscr{M}_{n,i}\mathscr{M}_{n,\binom{n}{2}}.
\end{equation*}
Therefore, there exist symmetric matrices $\boldsymbol{A}$ and $\boldsymbol{B}$ such that 
\begin{equation*}
\boldsymbol{U}^{\mathrm{T}}\mathscr{M}_{n,i}\boldsymbol{U}=
\mathrm{diag}(\boldsymbol{A},\boldsymbol{B})
\quad\text{ and }\quad \boldsymbol{U}^{\mathrm{T}}
\mathscr{M}_{n,\binom{n}{2}-i}\boldsymbol{U}=\mathrm{diag}
(\boldsymbol{A},-\boldsymbol{B}).
\end{equation*}
Thus the eigenvalues of $\mathscr{M}_{n,i}$ and $\mathscr{M}_{n,\binom{n}{2}-i}$ can be mutually determined for $i=1, \ldots, \binom{n}{2}$.

\section{Representation theory of the operators}
Clearly, $\mathscr{Y}_{n,i}$ acts not just on $\mathbb{C}\mathfrak{S}_n$, but on any of its modules, i.e., on any representations of $\mathfrak{S}_n$. Thus we can consider the eigenvalues of the $\mathscr{Y}_{n,i}$-action on the Specht modules, and so the eigenvalues of $\mathscr{M}_{n,i}$ can be determined explicitly by the $\mathscr{Y}_{n,i}$-action on the Specht modules. 

Recall that a \textit{partition} $\lambda=(\lambda_1, \lambda_2,\ldots)$ of $n$, denote $\lambda\vdash n$,  is a  weakly decreasing sequence of positive integers such that $|\lambda|=\sum_{i}|\lambda_i|$, and we write $\ell(\lambda)$ the length of $\lambda$, i.e. the number of parts of $\lambda$. The \textit{(Young) diagram} of a partition $\lambda\vdash n$ may be formally defined as the set of points $(i,j)\in \mathbb{Z}^2$ such that $1\leq j\leq \lambda_i$. In drawing such diagrams we shall adopt the convention that the first coordinate $i$ (the \textit{row index}) increasing as one goes downwards, and the second coordinate $j$ (the \textit{column index}) increasing as one goes from left to right. For example, the diagram of the partition $\lambda=(3,2,2,1)$ is
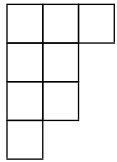
\begin{equation*}
\diagram(&&\cr&\cr&\cr)
\end{equation*}
We shall usually denote the diagram of a partition $\lambda$ by the same symbol $\lambda$.

For $\lambda\vdash n$, a $\lambda$-tableau $\mathfrak{t}$ is obtained from $\lambda$ by inserting the number $1,2,\ldots,n$ into its boxes with each integer occurring exactly once.  We write $\mathrm{shape}(\mathfrak{t})=\lambda$ when $\mathfrak{t}$ is a $\lambda$-tableau. A $\lambda$-tableau $\mathfrak{t}$ is {\it row} (resp. {\it column}) {\it standard} if the entries of $\mathfrak{t}$ are strictly increasing in each row (resp. column).  A $\lambda$-tableau is {\it  standard} if it is both row standard and column standard and we denote by $\mathrm{std}(\lambda)$ the set of all standard $\lambda$-tableaux. For example, let $\mathfrak{t}^{\lambda}$ (resp. $\mathfrak{t}_{\lambda}$) be the
$\lambda$-tableau with the numbers $1,2,\dots,n$ entered in
order first along the rows (resp. columns) of the boxes of $\lambda$. Then $\mathfrak{t}^{\lambda}$ and $\mathfrak{t}_{\lambda}$ are standard $\bgl$-tableaux:
\begin{equation*}
\mathfrak{t}^{\lambda}=\diagram(1&2&3\cr4&5\cr6&7\cr8),\qquad
\mathfrak{t}_{\lambda}=\diagram(1&5&8\cr2&6\cr3&7\cr4), \qquad
\mathfrak{a}=\diagram(2&3&4\cr1&8\cr6&7\cr5),\qquad \mathfrak{b}=\diagram(2&3&1\cr4&5\cr6&7\cr8),\qquad
\mathfrak{c}=\diagram(1&4&5\cr2&6\cr3&8\cr7),
\end{equation*}
where $\mathfrak{a}$ (resp. $\mathfrak{b}$) is row (resp. column) standard $\lambda$-tableau and $\mathfrak{c}$ is $\lambda$-standard tableau.

Note that $\mathfrak{S}_n$ acts naturally on the $\lambda$-tableaux by permute its entries.
We say that $\lambda$-tableaux $\mathfrak{a}$ and $\mathfrak{b}$ are \textit{row equivalent} if the numbers in the $i$-th rows of $\mathfrak{a}$ and $\mathfrak{b}$ are same for each $i=1, \ldots, \ell(\lambda)$. We denote by $\bar{\mathfrak{a}}$ the row equivalent class represented by $\mathfrak{a}$, which is called a \textit{tabloid}. For a $\lambda$-tableau $\mathfrak{t}$, denote by $\mathcal{C}_{\mathfrak{t}}$ the column stabilizer of $\mathfrak{t}$, that is,
\begin{equation*}
  \mathcal{C}_{\mathfrak{t}}:=\{\sigma\in \mathfrak{S}_n|\text{ the numbers in each column of $\mathfrak{t}$ is $\sigma$-invariant}\}.
\end{equation*}

\begin{definition}Let $\lambda$ be a partition of $n$ and $\mathfrak{t}$ a $\lambda$-tableau. The \textit{polytabloid} associated to $\mathfrak{t}$ is defined as
\begin{equation*}
  \varepsilon_{\mathfrak{t}}:=\sum_{\sigma\in
  \mathcal{C}_{\mathfrak{t}}}\mathrm{sgn}(\sigma)
  \overline{\sigma(\mathfrak{t})}.
\end{equation*}
The Specht module $S^{\lambda}$ associated to $\lambda\vdash n$ is defined as
\begin{equation*}
  S^{\lambda}:=\mathrm{span}_{\mathbb{C}}
  \{\varepsilon_{\mathfrak{t}}|\mathfrak{t}\in\mathrm{std}(\lambda)\}.
\end{equation*}
\end{definition}

It is well-known that $\{\varepsilon_{\mathfrak{t}}|\mathfrak{t}\in\mathrm{std}(\lambda)\}$ is a $\mathbb{C}$-basis of $S^{\lambda}$, which is called the standard basis of $S^{\lambda}$. Indeed, this is true for any field $\mathbb{F}$. Moreover, if the characteristic of $\mathbb{F}$ is zero then $S^{\lambda}$ is irreducible, and $\{S^{\lambda}|\lambda\vdash n\}$ is the complete set of non-isomorphic irreducible representations of $\mathbb{C}\mathfrak{S}_n$ and $H_n(q)$, that is, one can
decompose $\mathbb{C}\mathfrak{S}_n$:
\begin{equation}\label{Equ:Sn-decom-module}
 \mathbb{C}\mathfrak{S}_n\cong \bigoplus_{\lambda\vdash n}
 h_{\lambda}S^{\lambda}\cong \bigoplus_{ \mathfrak{t}\text{ a standard tableau}} S^{\mathrm{shape}(\mathfrak{t})},
\end{equation}
where $h_{\lambda}$ is the number of standard $\lambda$-tableaux, which is given by the Frame--Robinson--Thrall formula (see \cite{Frame-RT}), and $\mathrm{shape}(\mathfrak{t})=\lambda$ when $\mathfrak{t}$ is a $\lambda$-standard tableau.
Hence, there is a bijection between copies of simple modules and
standard tableaux, and one can associate each
standard tableau with a copy of a Specht module. Moreover, since Specht modules are simple, Schur's lemma shows that their morphisms (shufflings, for example) have a single eigenvalue when restricted to a Specht module.

Note that $\mathscr{Y}_{n,i}$ ($i=1,\ldots, \binom{n}{2}$)
act not just on $\mathbb{C}\mathfrak{S}_n$, but on any of its modules, i.e., on any representation of $\mathfrak{S}_n$. Thus the question about eigenvalues can be asked for each
representation of $\mathfrak{S}_n$, in particular for the Specht modules, which may help us to understand the eigenvalues of $\mathscr{M}_{n,i}$.

\begin{example}\label{Exam:n=3-Specht}
For $\lambda=(2,1)\vdash3$, we have
  \begin{eqnarray*}
  &&\mathrm{std}(\lambda)=\left\{\mathfrak{a}={\diagram(1&2\cr3)}, \quad \mathfrak{b}={\diagram(1&3\cr2)}\right\};\\
  &&\mathcal{C}_{\mathfrak{a}}=\{1, s_1s_2s_1\}, \qquad \varepsilon_{\mathfrak{a}}=\bar{\mathfrak{a}}-s_1s_2s_1\bar{\mathfrak{a}}
  =(1\,2|3)-(2\,3|1);\\
  &&\mathcal{C}_{\mathfrak{b}}=\{1, s_1\},\quad \varepsilon_{\mathfrak{b}}=\bar{\mathfrak{b}}-s_1\bar{\mathfrak{b}}
  =(1\,3|2)-(2\,3|1),
  \end{eqnarray*}
 where $(a\,b|c)$ means the set of the numbers in the first (resp. second) row is $\{a,b\}$ (resp. $\{c\}$). Thus
 \begin{equation*}
   S^{(2,1)}=\mathrm{span}_{\mathbb{C}}\{\varepsilon_{\mathfrak{a}},
   \varepsilon_{\mathfrak{b}}\}.
 \end{equation*}
 Furthermore, calculations show
 \begin{eqnarray*}
  \mathscr{Y}_{3,1}(\varepsilon_{\mathfrak{a}},\varepsilon_{\mathfrak{b}})
  &=&(\varepsilon_{\mathfrak{a}},\varepsilon_{\mathfrak{b}})
  \left(\begin{array}{cc}1&1\\&-1
  \end{array}\right);\\
\mathscr{Y}_{3,2}(\varepsilon_{\mathfrak{a}},\varepsilon_{\mathfrak{b}})
  &=&(\varepsilon_{\mathfrak{a}},\varepsilon_{\mathfrak{b}})
  \mathrm{diag}(-1,-1);\\
  \mathscr{Y}_{3,3}(\varepsilon_{\mathfrak{a}},\varepsilon_{\mathfrak{b}})
  &=&(\varepsilon_{\mathfrak{a}},\varepsilon_{\mathfrak{b}})
 \left(\begin{array}{cc}-1&-1\\&1
  \end{array}\right).\end{eqnarray*}
Finally, Equ.~\eqref{Equ:Sn-decom-module} shows
\begin{equation*}
  \mathbb{C}\mathfrak{S}_3\cong S^{(3)}\oplus 2S^{(2,1)}\oplus S^{(1^{3})}.
\end{equation*}
As a consequence, the eigenvalues of $\mathscr{M}_{3,1}$ are $2,1^2, (-1)^2,-2,$; the eigenvalues of $\mathscr{M}_{3,2}$ are $2^2,(-1)^4$; the eigenvalues of $\mathscr{M}_{3,3}$ are $1^3, (-1)^3$, which recovers the assertions in Example~\ref{Exam:n=3-eigenvalues}.
\end{example}

\begin{example}\label{Exam:n=4-Specht}Now we consider $\mathfrak{S}_4$. Firstly,  for $(3,1)\vdash 4$, we have  \begin{eqnarray*}
  &&\mathrm{std}((3,1))=\left\{\mathfrak{a}_1={\diagram(1&2&3\cr4)},\quad \mathfrak{a}_2={\diagram(1&2&4\cr3)}, \quad \mathfrak{a}_3={\diagram(1&3&4\cr2)}\right\};\end{eqnarray*}
\begin{eqnarray*}
 && \mathcal{C}_{\mathfrak{a}_1}=\{1, (1\,4)\},\quad \varepsilon_{\mathfrak{a}_1}=(1\,2\,3|4)-(2\,3\,4|1);\\
 && \mathcal{C}_{\mathfrak{a}_2}=\{1, (1\,3)\},\quad \varepsilon_{\mathfrak{a}_2}=(1\,2\,4|3)-(2\,3\,4|1);\\
&& \mathcal{C}_{\mathfrak{a}_3}=\{1, (1\,2)\}, \quad \varepsilon_{\mathfrak{a}_3}=(1\,3\,4|2)-(2\,3\,4|1).
\end{eqnarray*}
Thus
\begin{eqnarray*}
 \mathscr{Y}_{4,1}(\varepsilon_{\mathfrak{a}_1}, \varepsilon_{\mathfrak{a}_2},\varepsilon_{\mathfrak{a}_3})=
 (\varepsilon_{\mathfrak{a}_1}, \varepsilon_{\mathfrak{a}_2},\varepsilon_{\mathfrak{a}_3})
 \left(\begin{array}{ccc}
  2&1&0\\1&1&1\\-1&0&0
 \end{array}\right),
  \end{eqnarray*}
which has eigenvalues $1$, $1+\sqrt{2}$, $1-\sqrt{2}$.

Secondly, for $(2^2)\vdash 4$, we have
  \begin{eqnarray*}
  &&\mathrm{std}((2^2))=\left\{\mathfrak{b}_1={\diagram(1&2\cr3&4)},\quad \mathfrak{b}_2={\diagram(1&3\cr 2&4)}\,\right\};\end{eqnarray*}
 \begin{eqnarray*} &&\mathcal{C}_{\mathfrak{b}_1}=\{1, (1\,3), (2\,4), (1\,3)(2\,4)\},\quad \varepsilon_{\mathfrak{b}_1}=(1\,2|3\,4)-(2\,3|1\,4)-(1\,4|2\,3)+(3\,4|1\,2);\\
  &&\mathcal{C}_{\mathfrak{b}_2}=\{1, (1\,2), (3\,4), (1\,2)(3\,4)\},\quad \varepsilon_{\mathfrak{b}_2}=(1\,3|2\,4)-(2\,3|1\,4)-
(1\,4|2\,3)+(2\,4|1\,3).
\end{eqnarray*}
Thus
  \begin{eqnarray*}
\mathscr{Y}_{4,1}(\varepsilon_{\mathfrak{b}_1}, \varepsilon_{\mathfrak{b}_2})=(\varepsilon_{\mathfrak{b}_1}, \varepsilon_{\mathfrak{b}_2})\left(\begin{array}{cc}
    2&1\\-1&-2
  \end{array}\right),
  \end{eqnarray*}
which has eigenvalues $\pm\sqrt{3}$.

Finally, for $(2,1^2)\vdash 4$, we have
\begin{eqnarray*}
  &&\mathrm{std}((2,1^2))=\left\{\mathfrak{c}_1={\diagram(1&2\cr3\cr4)},
\quad \mathfrak{c}_2={\diagram(1&3\cr2\cr4)}, \quad \mathfrak{c}_3={\diagram(1&4\cr2\cr3)}\right\};
\end{eqnarray*}
\begin{eqnarray*}  &&\mathcal{C}_{\mathfrak{c}_1}=\mathfrak{S}_{\{1,3,4\}}=\{(1), (1\,3), (1\,4), (3\,4), (1\,3\,4), (1\,4\,3)\},\\&&
\varepsilon_{\mathfrak{c}_1}=(1\,2|3|4)-(2\,3|1|4)-(2\,4|3|1)-(1\,2|4|3)
+(2\,3|4|1)+(2\,4|1|3);\\
&& \mathcal{C}_{\mathfrak{c}_2}=\mathfrak{S}_{\{1,2,4\}}=\{(1), (1\,2), (1,4), (2\,4), (1\,2\,4), (1\,4\,2)\},\\&&
\varepsilon_{\mathfrak{c}_2}=(1\,3|2|4)-(2\,3|1|4)-(3\,4|2|1)-(1\,3|4|2)
+(2\,3|4|1)+(3\,4|1|2);\\
&&\mathcal{C}_{\mathfrak{c}_3}=\mathfrak{S}_{\{1,2,3\}}=\{(1), (1\,2), (1\,3), (2\,3), (1\,2\,3), (1\,3\,2)\}\\&&
\varepsilon_{\mathfrak{c}_3}=(1\,4|2|3)-(2\,4|1|3)-(3\,4|2|1)-(1\,4|3|2)
+(2\,4|3|1)+(3\,4|1|2).
  \end{eqnarray*}
Thus \begin{eqnarray*}
 \mathscr{Y}_{4,1}(\varepsilon_{\mathfrak{c}_1}, \varepsilon_{\mathfrak{c}_2},\varepsilon_{\mathfrak{c}_3})=
 (\varepsilon_{\mathfrak{c}_1}, \varepsilon_{\mathfrak{c}_2},\varepsilon_{\mathfrak{c}_3})
 \left(\begin{array}{ccc}
  0&1&0\\0&-1&1\\1&1&-2
 \end{array}\right),
  \end{eqnarray*}
which has eigenvalues $-1, -1+\sqrt{2}, -1-\sqrt{2}$.

Now applying Equ.~\eqref{Equ:Sn-decom-module}, we have
\begin{equation*}
  \mathbb{C}\mathfrak{S}_4\cong S^{(4)}\oplus 3S^{(3,1)}\oplus 2S^{2^2}\oplus 3S^{(2,1^2)}\oplus S^{(1^{3})},
\end{equation*}
which helps us to obtain the eigenvalues $\mathscr{M}_{4,1}$ via the $\mathscr{Y}_{4,1}$-action on the Specht modules.
\end{example}

Let us remark that the eigenvalues of $\mathscr{M}_{n,i}$ can be obtained via the eigenvalues of the $\mathscr{Y}_{n,i}$-action on the Specht modules. Moreover, the eigenvalues of $\mathscr{M}^q_{n,i}$ can be obtained via the eigenvalues of the $\mathscr{Y}_{n,i}$-action on the Specht modules.

\end{CJK*}
\end{document}